\documentclass{aptpub}
\usepackage{mathrsfs,epsfig,color,pdfsync}

\authornames{Takis Konstantopoulos, Andreas E. Kyprianou, Paavo Salminen} 

\shorttitle{On the excursions of reflected local time processes} 

\def\Z{\mathbb{Z}}

\def\R{\mathbb{R}}

\def\P{\mathbb{P}}
\def\E{\mathbb{E}}
\def\FF{\mathcal{F}}
\def\GG{\mathcal{G}}
\def\BB{\mathcal{B}}

\def\CC{\mathsf{C}}

\renewcommand{\phi}{\varphi}
\renewcommand{\epsilon}{\varepsilon}

\newcommand{\1}{{\text{\Large $\mathfrak 1$}}}
\newcommand{\comp}{\raisebox{0.1ex}{\scriptsize $\circ$}}
\newcommand{\cadlag}{{c\`adl\`ag} }

\def\ZZZ{\mathcal{Z}}

\newcommand{\QED}{\hfill $\square$}

\begin{document}

\title{On the excursions of reflected local time \\ processes and stochastic fluid queues} 

\authorone[Heriot-Watt University]{Takis Konstantopoulos} 
\vspace{-0.8cm}
\addressone{
School of Mathematical Sciences, Heriot-Watt University, Edinburgh, EH14 4AS, UK} 
\authortwo[University of Bath]{Andreas E. Kyprianou} 
\addresstwo{Department of Mathematical Sciences, University of Bath, Claverton Down, Bath, BA2 7AY, UK}
\vspace{-0.8cm}
\authorthree[ \AA bo Akademi
University]{Paavo Salminen}
\addressthree{Department of Mathematics, \AA bo Akademi
University, Turku, FIN-20500, Finland}
\renewcommand{\thefootnote}{\arabic{footnote}}

\begin{abstract}
This paper extends previous work by the authors.
We consider the local time process of a strong Markov process, add negative
drift, and reflect it \`a la Skorokhod. The resulting process is used to model a fluid queue.  We derive an expression
for the joint law of the duration of an excursion, the maximum
value of the process on it, and the time distance between
successive excursions. We work with a properly constructed
stationary version of the process.
Examples are also given in the paper.





\end{abstract}
\section{Introduction}
\label{Intro}
Consider a stationary strong Markov process $X=(X_t, t \in \R)$,
defined on some filtered probability space 
$(\Omega, \FF, P, (\FF_t, t \in \R))$,
with values in $\R_+$, a.s.\ \cadlag paths, and adapted to $(\FF_t)$.
In this paper, the local time $L$ of the process $X$ at $x=0$
is considered an $(\FF_t)$-adapted stationary random measure 
that regenerates jointly with $X$ at every (stopping)
time that $X$ hits $0$. More precisely:
\begin{itemize}\em
\item[{\bf (A1)}]
$L$ assigns a nonnegative random
variable $L(B,\omega)$ to each $B \in \BB(\R)$ such that $L(\cdot, \omega)$
is a Radon measure for each $\omega \in \Omega$. 
\item[{\bf (A2)}]
For any a.s.\ finite $(\FF_t)$-stopping time $T$ at which $X_T=0$,
the process $((X_{T+t}, L(T, T+t)), t \ge 0)$ is independent of $\FF_T$.
\end{itemize}
We take the broader perspective with regard to the process $L$ and we allow for the case that it is a local time of an irregular point (in which case $L$ has discontinuous paths) as well as the case that $0$ is a sticky point 
(in which case $L$ is absolutely continuous with respect to the Lebesgue measure with density $c\1(X_t = 0)$ for some $c>0$). 
We refer to \cite[Chap.\ IV]{BER} (in particular Corollary 6), \cite[Chap.\ 6]{K}, and \cite[\S V.3]{BG}
for further discussion.
For each $s \in \R$ define the inverse local time
process with respect to $t $ by
\begin{equation}
\label{Linv}
L_{s; u}^{-1} := \inf\{ t>0:~ L[s, s+t] > u \},
\quad u \ge 0.
\end{equation}
What is important is that, owing to this definition, the inverse
of the cumulative local time is a L\'evy process in the following sense:
\begin{lemma}
\label{basic}
If $L$ is continuous then for every a.s.\ finite $(\FF_t)$-stopping time
$T$ such that $X_T=0$, the process
$(L^{-1}_{T;u}, u \ge 0)$ is a subordinator
with $L^{-1}_{T;0}=0$.
\end{lemma}
If $L$ is not continuous, that is to say if $0$ is an irregular point for $X$, then this Lemma is taken as an
additional requirement to the definition of $L$. This is easily arranged by choosing $L$ to be a modification of
the counting process on $\ZZZ$, the discrete set of times that $X$ visits $0$, so that the inverse is a subordinator.
To do this, we assign, to each element of $\ZZZ$,
an i.i.d.,  unit-mean exponentially distributed weight. Then  let
the local time on an interval $I$ to be the sum of all the weights of
the points of $\ZZZ$ in $I$.

We summarise this as an assumption, in addition to (A1)-(A2) above:
\begin{enumerate}
\em
\item[\bf (A3)]
If $L$ is discontinuous then we require that 
for every a.s.\ finite $(\FF_t)$-stopping time
$T$ such that $X_T=0$, the process
$(L^{-1}_{T;u}, u \ge 0)$ is a subordinator.
\end{enumerate}
We will also need the following assumption:
\begin{enumerate}
\em
\item[\bf (A4)]
The stationary random measure $L$ has finite rate not
exceeding $1$, i.e.
\[
EL(0,t) = \mu\,t,
\]
where $0<\mu<1.$
\end{enumerate}

Then, as in \cite{MNS}, \cite{KoSa}, \cite{Si}, and \cite{KKSS},  we define a 
stationary process $Q=(Q_t, t \in\mathbb{R})$ by
\begin{equation}
\label{Q}
Q_t = \sup_{-\infty < s \le t} \left\{ L(s,t]-(t-s)\right\}, 
\quad t \in \R.
\end{equation}
Furthermore, $Q$ is ergodic (its invariant $\sigma$-field is trivial.)
Notice that $Q$ also satisfies, pathwise,
\begin{align}
Q_t &= Q_s + L(s,t]-(t-s)  - \inf_{s \le r \le t}
\big( Q_s + L(s,r]-(r-s) \big) \wedge 0,
\nonumber
\\
&= \sup_{s \le r \le t} (L(r,t]-(t-r)) \vee
(Q_s + L(s,t]-(t-s))
\label{sss}
\end{align}
for all $-\infty < s < t < \infty$.
It is worth recalling \cite{KKSS} that if we consider \eqref{sss}
as a fixed point equation for $Q$ then process defined by \eqref{Q}
is the unique stationary and ergodic solution of \eqref{sss}.
A typical sample path of $Q$ is depicted in Figure \ref{typical} below.
It consists of isolated excursions away from zero (also called 
``busy periods''), followed by intervals of time at which $Q$ stays
at zero (called ``idle periods''). In this respect, the process $Q$ is thought of as the workload in a stochastic fluid queue. Amongst other things 
in \cite{KoSa}, \cite{Si}, and \cite{KKSS}, expressions are  derived
for the marginal distribution of $Q$ and the Laplace transform of the duration of a typical
idle and busy periods.

In this paper, we shall derive an expression for
the joint law of three random variables: the duration of a busy period,
the duration of an idle period, and the maximum of $Q$ over a busy period.


It is assumed, throughout, that $(\Omega, \FF, P)$ is endowed with
a $P$-preserving measurable flow $\theta_t:\Omega\to \Omega$, $t \in \R$,
with a measurable inverse $\theta_t^{-1} = \theta_{-t}$.
All stationary random processes and measures can be constructed on
$\Omega$ in such a way that the flow commutes with the natural shift, e.g., 
$Q_t (\theta_s \comp \omega) = Q_{t+s}(\omega)$,
and $L(B, \theta_s\omega) = L(B+s, \omega)$, for all
$s,t \in \R$, Borel sets $B \subset \R$, and $\omega \in \Omega$.
The flow will be explicitly used  in
Section \ref{cycleform} to obtain distributions conditional on
observing a positive (or a zero) value of $Q_0$.

\section{A closer look at the reflected process}\label{closerlook}
Consider now any a.s.\ finite  $(\FF_t)$-stopping time $T$,  
such that $X_T=0$.
Then $(L^{-1}_{T; t}, ~ t \ge 0)$ is a subordinator starting from zero
(owing to Lemma \ref{basic} or assumption (A3)) with law that does not
depend on $T$. It turns out that the process of interest is $\Lambda_T = \{\Lambda_{T,t}: t\geq 0\}$, where
\begin{equation}
\label{LLL}
\Lambda_{T;t} = t - L^{-1}_{T; t} , \quad t\geq 0.
\end{equation}
Note that, irrespective of $T$, the process $\Lambda_T$ has 
the law of the same bounded variation, spectrally negative L\'evy process which 
is issued from the origin at time zero. 
By (A4), $L$ has rate $\mu < 1$; 
hence $E \Lambda_{T;1} = 1-\frac{1}{\mu}< 0$.  Since $L^{-1}_{T;t}$ is a subordinator, it has 
a well-defined, possibly nonzero, drift.
If this drift is larger than or equal to unity then $-\Lambda_T$
is a subordinator and, as it will turn out, this is a trivial case. 

We therefore  assume in the sequel that the drift of $L^{-1}_T$ is less
than unity or, equivalently, that 
\begin{enumerate}
\em
\item[\bf (A5)]
The drift $\delta_\Lambda$ of the process $\Lambda$ defined by 
\eqref{LLL} 
is strictly positive.
\end{enumerate}

Under this assumption,
the point $0$ is irregular for $(-\infty,0)$ for $\Lambda_T$
(this follows as a standard results for bounded variation spectrally 
negative L\'evy processes, see Bertoin \cite[Chap.\ VII]{BER}.)

In addition, under (A5), it is clear that the time taken for $\Lambda_T$ to
first enter $(-\infty, 0)$ is almost surely strictly positive.
It will be shown below (Lemma \ref{BI}) that this implies that the
excursions of the process $Q$, i.e.~the busy periods, 
have strictly positive Lebesgue length with probability one.
\begin{figure}[h]
\begin{center}
\includegraphics[width=10cm]{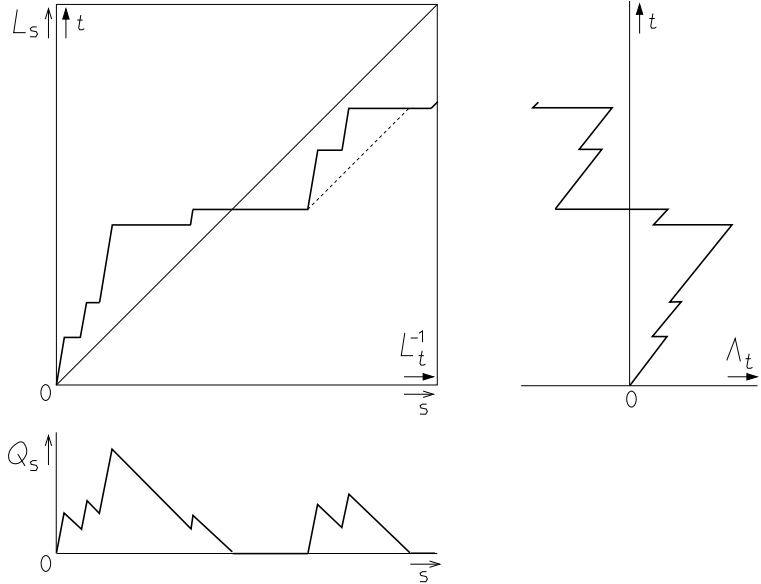}
\caption{\em The construction of the process $(\Lambda_{T;t},~t \ge 0)$
and related processes, assuming that $T=0$. Note that $\Lambda$ may have
countably many jumps on finite intervals.} 
\label{lambda}
\end{center}
\end{figure}
It can be intuitively seen, via a geometric argument
involving the reflection of the space-time path of
$\Lambda_T$ about the diagonal (see Figure \ref{lambda}), that the time taken for $\Lambda_T$ to
first enter $(-\infty, 0)$ is almost surely equal to the length of
the excursion of $Q$ started at time $T$.
 
In this light, note also that $\Lambda_T$ cannot creep downwards 
because it is spectrally negative with paths of bounded variation 
(cf. Bertoin \cite[Chap.\ VII]{BER}). 
Hence the overshoot at first passage of $\Lambda_T$ into $(-\infty, 0)$
is almost surely strictly positive. It will turn out (Lemma \ref{basic})
that this overshoot agrees with the idle period following the aforementioned
excursion of $Q$.

The above analysis implies that, on finite intervals of time, $Q$ has finitely many excursions (busy periods)
separated by positive-length idle periods. 
Denote by
\[
\cdots < g(-1) < g(0) < g(1) < g(2) < \cdots
\]
the beginnings of the idle periods and by
\[
\cdots < d(-1) < d(0) < d(1) < d(2) < \cdots
\]
their ends, see Figure \ref{typical}.
We choose the indexing so that $g(0) \le 0 < g(1)$.
Let $N_g$ (respectively, $N_d$) be the point process with points 
$\{g(n): n \in \Z\}$ (respectively, $\{d(n): n \in \Z\}$).
As $Q$ is a stationary process, 
$N_g$ and $N_d$ are jointly stationary with finite, nonzero, intensity
\cite{KKSS} denoted by $\lambda$ (an expression for which
is given by \eqref{commonrate} and is derived in \S \ref{Rates} below). 
\begin{figure}[h]
\begin{center}
\includegraphics[width=\columnwidth]{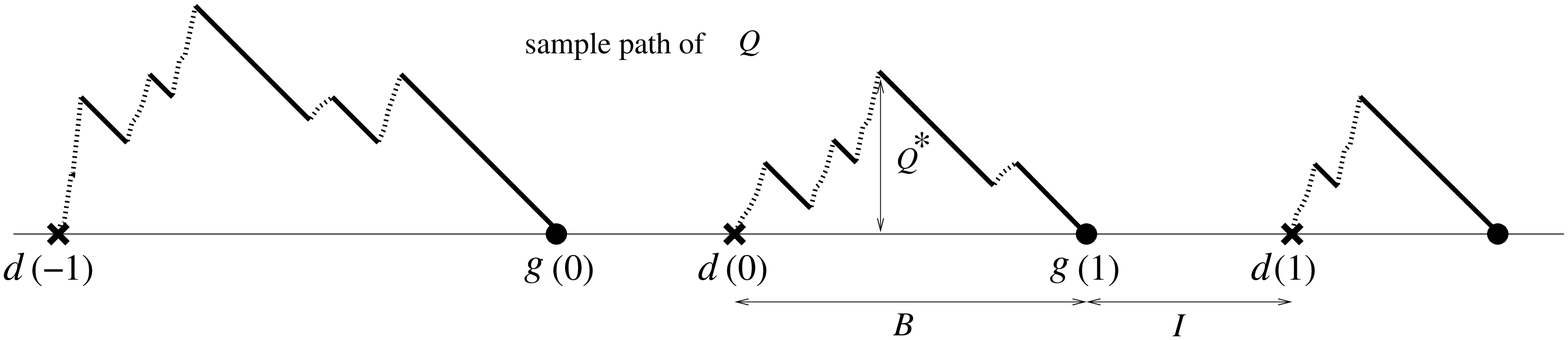}
\end{center}
\caption{\em The definition of $g(n)$ and $d(n)$. By convention,
the origin of time is between $g(0)$ and $g(1)$, under the original 
measure $P$. Under $P_d$, the origin of time is at $d(0)$. Under $P_g$,
the origin of time is at $g(0)$. The random variable $Q^*$ is the maximum
deviation from $0$ of $Q$ within the typical busy period.} 
\label{typical}
\end{figure}
Corresponding to point processes $N_g$, $N_d$ we have the 
Palm probabilities $P_g$, $P_d$, respectively.
Let us consider $Q$ under the measure $P_d$. Then
$P_d(d(0)=0)=1$, i.e.\ the origin of time is placed at
the beginning of a busy period. 
By the strong Markov property, the ``cycles''
\[
\CC_n:= \left\{Q_t:~ d(n) \le t < d(n+1)\right\}, \quad n \in \Z,
\]
are i.i.d.\ under measure $P_d$.
In particular, 
the pairs of random variables
\[
\big(g(n+1)-d(n),~ d(n+1)-g(n+1)\big), \quad n \in \Z, 
\]
are i.i.d.\ under $P_d$. 
Consider the triple 
\begin{equation}
\label{triple}
(B, I, Q^*) := \left(g(1)-d(0),~ d(1)-g(1),~ \sup_{d(0) < t < g(1)} Q_t \right),
\end{equation}
which is a function of $\CC_0$.
We are primarily interested in the $P_d$-law of $(B,I,Q^*)$
Since, under $P_d$, the origin of time is placed at $d(0)$,
we interpret $B, I,Q^*$ as the typical busy period, the typical
idle period, and the maximum value of $Q$ over a typical
busy period, respectively.

The next lemma is proved in \cite{KKSS}:
\begin{lemma}
\label{Dd}
Let $D =\inf\{t >0:~ X_t =0\}$ and $d=\inf\{t >0:~ Q_t >0\}$.
Then $d=D$ a.s.\ on $\{Q =0\}$.
\end{lemma}

We now obtain an alternative expression for $B= g(1) -d(0)$
and $I = d(1)-g(1)$
in terms of the inverse local time.
\begin{lemma} We have that
\label{BI}
\begin{align}
B& = g(1)-d(0) = \inf\{u>0: ~ L^{-1}_{d(0);u} > u \},
\label{B}
\\
B+I &= d(1)-d(0) =  L^{-1}_{d(0); g(1)-d(0)}.
\label{I}
\end{align}
\end{lemma}
\begin{proof}
Since $d(0)$ is the end of an idle period (and the beginning
of a busy period), we have $Q_{d(0)-}=0$.
Using then expression (\ref{sss}) we obtain
\[
Q_t =  L[d(0), t] - (t-d(0)), \quad d(0) \le t < g(1),
\]
which gives
\[
B= g(1) -d(0) 
= \inf\{t > 0:~ L[d(0), d(0)+t] =t\}.
\]
Consider now $L^{-1}_{d(0);u}$, defined by \eqref{Linv}.
By Lemma \ref{Dd}, $d(0)$ is a point of increase of the function
$t \mapsto L[d(0),d(0)+t]$.
Hence $g(1) > d(0)$.
Also, when $L[d(0),d(0)+t]-t$ decreases, it does so continuously.
Therefore,
\[
B = \inf\{t > 0:~ L[d(0),d(0)+t]<t\}.
\]
Notice also that, for all $t,x > 0$,
\[
L[d(0),d(0)+t] < x \iff t< L^{-1}_{d(0); x-},
\]
where 
$L^{-1}_{d(0); x-} = \lim_{\epsilon \downarrow 0}L^{-1}_{d(0); x-\epsilon}$.
It follows that,
\begin{align*}
B &= \inf\{t >0:~ t< L^{-1}_{d(0); t-}\}
\\
&= \inf\{t >0:~ L^{-1}_{d(0); t} > t\},
\end{align*}
by the right continuity of $t \mapsto L^{-1}_{d(0); t}$.
To prove the expression for $B+I$, notice that $L$ does
not charge the interval $[g(1), d(1))$ because, by definition,
$Q$ is zero for all $t$ in this interval.
\QED
\end{proof}


Henceforth it will be convenient to work with the process $\Lambda= (\Lambda_t, t\geq 0)$ where
\[
\Lambda_t := t - L^{-1}_{d(0);t}, \quad t \ge 0.
\]
Note also that $d(0)$ is an $(\FF_t)$-stopping time at which $X$ takes the value 0 and hence in 
our earlier notation $\Lambda_t = \Lambda_{d(0);t}$. 

{From} the expression (\ref{B}), and as discussed in the introduction of  Section \ref{closerlook}, we see that $B$ is simply the first
time at which $\Lambda$ enters $(-\infty, 0)$,
\begin{equation}
\label{Bnew}
B= \inf\{t>0:~ \Lambda_t < 0\},
\end{equation}
which is necessarily strictly positive thanks to the irregularity of $0$ for $(-\infty,0)$ for $\Lambda$.
{From} \eqref{B} and \eqref{I} we see that 
\begin{equation}
\label{Inew}
I= L^{-1}_{d(0);g(1)}-B = L^{-1}_{d(0);B} -B
= - \Lambda_B,
\end{equation}
i.e.\ $I$ is, in absolute value, equal to the value of $\Lambda$ at the first time it becomes
negative. Again we recall from the discussion at the beginning of Section \ref{closerlook}
that  $\Lambda$ cannot creep downwards and hence
$
 I >0
$
almost surely.


Consider now the random variable $Q^*=\sup_{d(0) < t < g(1)} Q_t$.
If we define
\begin{equation}
\label{tau}
\tau_x := \inf\{t>0:~ \Lambda_t > x\}
= \inf\{t>0:~ \Lambda_t = x\}
\end{equation}
we immediately see that
\begin{equation}
\label{Mnew}
\{Q^*< x\} = \{ B< \tau_x \}. 
\end{equation}

\section{The triple law}
Recall that $P_d$ is the Palm probability with respect to the point
process $\{d(n), n \in \Z\}$.
The function
\[
H(\alpha, \beta, x) 
= E_d\left[e^{-\alpha B - \beta I} \1(Q^* \le x) \right]
\]
characterizes the joint law of the triple $(B,I,Q^*)$ under $P_d$.
Since $P_d(d(0)=0)=1$, we have that 
\begin{equation}
\label{LAMBDA}
\text{$\Lambda_t = t-L^{-1}_{0;t}$, 
with $\Lambda_0 =0$, $P_d$-a.s.}
\end{equation}
Recalling the expressions \eqref{Bnew}, \eqref{Inew} and \eqref{Mnew}
for $B$, $I$ and $Q^*$, respectively, we write
\begin{equation}
\label{H}
H(\alpha, \beta, x)
= E_d\left[e^{-\alpha B +\beta\Lambda_B } \1(B \le \tau_x)\right].
\end{equation}

Since our primary object is the process $\Lambda$ defined in \eqref{LAMBDA},
and in view of \eqref{LLL} and \eqref{H}, it makes sense to consider
the process on its canonical probability space and denote its law
by $\P_0$.
Then
\begin{equation}
\label{HH}
H(\alpha, \beta, x)
= \E_0 \left[e^{-\alpha B +\beta \Lambda_B} \1(B < \tau_x)\right].
\end{equation}
The latter function may now be expressed in terms of so-called 
scale functions for spectrally negative L\'evy processes. 
To define the latter, let
\[
\psi_\Lambda(\theta): = \log \E_0  e^{\theta \Lambda_1}, \, \theta\geq 0,
\]
 be the Laplace exponent of $\Lambda$ under $\mathbb{P}_0 $. 
Then the, so-called, $q$-scale function for 
$(\Lambda, \mathbb{P}_0 )$, 
denoted by $W^{(q)}(x)$, 
satisfies $W^{(q)}(x) = 0$ for $x<0$ and 
on $[0,\infty)$ it is the unique continuous 
(right continuous at the origin) monotone increasing function whose Laplace 
transform is given by 
\begin{equation}
\label{alpha}
\int_0 ^\infty e^{-\theta x} W^{(q)}(x) dx =
\frac{1}{\psi_\Lambda(\theta)-q}, \quad \text{for } \beta>\Phi_\Lambda(q),
\end{equation}
where 
\[
 \Phi_\Lambda(q) = \sup\{\theta\geq 0 : \psi_\Lambda(\theta) = q\}
\]
is the right inverse of $\psi_\Lambda$. (See for example the discussion in Chapter 9 of \cite{K}).

\begin{theorem} 
Let $\Lambda$ be the process defined by \eqref{LAMBDA}, $B$ its first entry time to
$(-\infty, 0)$ as in \eqref{Bnew}, and $\tau_x$ the first hitting time of $\{x\}$ as
in \eqref{tau}.
For $\alpha, \beta, x\geq 0$ we have
\begin{eqnarray}
\label{Hformula}
\nonumber      
H(\alpha, \beta, x)
&=& \E_0 \left[e^{-\alpha B +\beta \Lambda_B} \1(B< \tau_x)\right]\\
&=& 1 - \frac{1}{\delta_\Lambda}
~
\frac{1 + (\alpha-\psi_\Lambda(\beta))
\int_0 ^x e^{-\beta y} W^{(\alpha)}(y) dy}
{e^{-\beta x} W^{(\alpha)}(x)}.
\end{eqnarray}
\end{theorem}
\begin{proof}
Let $\GG_t:= \sigma(\Lambda_s, s \le t)$ and define,
for all $\beta \ge 0$, the exponential $(\GG_t)$-martingale
\[
M^\beta_t := e^{\beta \Lambda_t - \psi_\Lambda(\beta)t}, \quad t \ge 0.
\]
Let, on the canonical space of $\Lambda$, $\P^\beta_0$ be a probability
measure, absolutely continuous with respect to $\P_0$ on $\GG_t$ for 
each $t$, with Radon-Nikod\'ym derivative
\[
\left. \frac{d\P_0 ^\beta}{d\P_0 } \right|_{\GG_t} := M^\beta_t.
\]
Notice that $\Lambda$ is still a L\'evy process under $\P^\beta_0$
with Laplace exponent
\begin{equation}
\label{psil}
\psi_\Lambda^\beta(\theta) =\log \E^\beta_0 e^{\theta \Lambda_1}
= \psi_\Lambda(\beta+\theta)-\psi_\Lambda(\beta).
\end{equation}
It is straightforward to check from the above formula that, under $\P^\beta_0$, 
$\Lambda$ is spectrally negative, with bounded variation paths
and drift coefficient equal to $\delta_\Lambda$.
Since on the stopped $\sigma$-field $\GG_B$ we have
$(d\P_0^\beta/d\P_0)\big|_{\GG_B} = M^\beta_B$,
we may substitute
\[
e^{\beta \Lambda_B} = M^\beta_B e^{\psi_\Lambda(\beta) B}
\]
in the equation \eqref{HH} for $H$ to obtain
\begin{align*}
H(\alpha, \beta, x) 
= \E_0  \left[M^\beta_B e^{\psi_\Lambda(\beta) B} e^{-\alpha B}
\1(B < \tau_x) \right]
= \E_0 ^\beta \left[ e^{-(\alpha-\psi_\Lambda(\beta)) B}
\1(B < \tau_x) \right].
\end{align*}
Let 
\[
q := \alpha-\psi_\Lambda(\beta),
\]
and assume that $q \ge 0$.
It follows from \cite[Thm,\ 8.1(iii)]{K} that 
\begin{equation}
\label{Hformula2}
H(\alpha, \beta, x) 
=\E_0 ^\beta \left[ e^{-qB} \1(B < \tau_x) \right]
= Z^{(q)}_\beta (0) - Z^{(q)}_\beta (x) \frac{W^{(q)}_\beta (0)}{W^{(q)}_\beta (x)},
\end{equation}
where $W^{(q)}_\beta$ is the $q$-scale function for $(\Lambda, \P^\beta_0)$
and $Z^{(q)}_\beta $ is given by 
\[
Z^{(q)}_\beta (x) = 1+q \int_0 ^x W^{(q)}_\beta (t) dt.
\]
It is easy to see \cite[Lemma 8.4]{K}
that the Laplace transform of $W^{(q)}_\beta(\cdot)$
is the Laplace transform of $W^{(q)}(\cdot)$ shifted by $\beta$ 
and this ensures that
\begin{equation}
 W^{(q)}_\beta (x) = e^{-\beta x} W^{(\alpha)}(x).
\label{Wq}
 \end{equation}
Moreover, since $\Lambda$ still has drift coefficient $\delta_\Lambda$ under $\P^\beta_0$,
\cite[Lemma 8.6]{K} tells us that, irrespective of the value of $q$ and $\beta$,  
$W^{(q)}_\beta(0) = 1/\delta_\Lambda$. 
Putting the pieces together, this gives us the desired expression for $\alpha \geq\psi_\Lambda(\beta)$. 
However \cite[Lemma 8.3]{K}, since $W^{(q)}(x)$ is analytic in $q$, 
the  condition on $\alpha$ can be relaxed to $\alpha\geq 0$ 
by using a straightforward analytic extension argument. 
\QED
\end{proof}

In view of \eqref{Bnew}, \eqref{Inew}, \eqref{Mnew}, and \eqref{H},
we get the following corollary.
\begin{corollary}[Joint law of typical $B$, $I$ and $Q^*$]
Assume that (A1)--(A5) hold. Then
the joint law of the length $B$ of a typical busy period,
the length $I$ of a typical idle period, and the maximum
$Q^*$ of $Q$ over the typical busy period is expressed by
the formula
\begin{equation}
\label{triplelaw}
E_d[e^{-\alpha B - \beta I} \1(Q^* \le x)] 
=
1 - \frac{1}{\delta_\Lambda}
~
\frac{1 + (\alpha-\psi_\Lambda(\beta))
\int_0 ^x e^{-\beta y} W^{(\alpha)}(y) dy}
{e^{-\beta x} W^{(\alpha)}(x)}
\end{equation}
where $\alpha, \beta, x\geq 0$.
\end{corollary}

\section{Marginal distributions}
Clearly, formula \eqref{triplelaw} can be used to extract more detailed information 
about typical  behaviour of $Q$.
Let us first derive the distribution (Laplace transform) of the
pair $(B,I)$ under the measure $P_d$. We have
\[
E_d [e^{-\alpha B} e^{-\beta I}] 
= \E_0  [ e^{-\alpha B} e^{\beta \Lambda_B}]
= \lim_{x \to \infty} H(\alpha, \beta, x).
\]
To derive the limit, let us temporarily assume that $q = \alpha- \psi_\Lambda(\beta)>0$ and $\beta\geq 0$. Consider \eqref{Hformula} in the
form \eqref{Hformula2}
and use the limiting result 
\[
\lim_{x \to \infty}
\frac{Z^{(q)}_\beta (x)}{W^{(q)}_\beta (x)}
= \frac{q}{\Phi^\beta_\Lambda(q)},
\]
from \cite[Exercise 8.5]{K}
where the function $\Phi^\beta_\Lambda$ is the right inverse of $\psi_\Lambda^\beta$. 
That is to say
\begin{align*}
\Phi_\Lambda^\beta(q) &= \sup\{\theta \geq 0 : \psi_\Lambda^\beta(\theta) = q\}\\
&=\sup\{\theta \geq 0 : \psi_\Lambda^\beta(\theta) = \alpha - \psi_\Lambda(\beta)\}\\
&=\sup\{\theta \geq 0 : \psi_\Lambda(\theta+\beta) = \alpha \}\\
&= \Phi_\Lambda(\alpha)-\beta. 
\end{align*}
This gives
\begin{equation}
E_d [e^{-\alpha B} e^{-\beta I}] 
=
1-\frac{1}{\delta_\Lambda}
\frac{\alpha-\psi_\Lambda(\beta)}{\Phi_\Lambda(\alpha) -\beta}.
\label{no-max}
\end{equation}
To remove the restriction that $\alpha> \psi_\Lambda(\beta)$ in (\ref{no-max}) and replace it instead by just $\alpha\geq 0$,
one may again proceed with an argument involving analytical extension taking care to note for the case that $\alpha  =\psi_\Lambda(\beta)$,
\[
 \lim_{|\alpha - \psi_\Lambda(\beta)|\rightarrow 0} \frac{\alpha-\psi_\Lambda(\beta)}{\Phi_\Lambda(\alpha) -\beta}=
\lim_{|\alpha - \psi_\Lambda(\beta)|\rightarrow 0} \frac{\psi^\beta_\Lambda(\Phi_\Lambda(\alpha) -\beta)}{\Phi_\Lambda(\alpha) -\beta}
= {\psi_\Lambda^{\beta}}'(0+) = \psi_\Lambda'(\beta).
\]
 
\subsection{Busy period}
Letting $\beta=0$ in (\ref{no-max}), we find the $P_d$-law of $B$. That is to say,
\[
E_d[e^{-\alpha B}] = 1-\frac{1}{\delta_\Lambda}
\frac{\alpha}{\Phi_\Lambda(\alpha)}.
\]
This formula is consistent with the result of 
\cite[Prop.\ 8]{KKSS} and, moreover, 
we see that the mean duration of the busy period is given by
\begin{equation}
\label{busym}
E_d[B] = \frac{1}{\delta_\Lambda \Phi_\Lambda(0)}.
\end{equation}


\subsection{Idle period}

To find the $P_d$-law of $I$ we need to set $\alpha=0$.
Recall however from the beginning of Section \ref{closerlook} that $E_d(\Lambda_1) <0$. This implies that $\Phi_\Lambda(0)>0$ and hence 
we have 
\begin{equation}
\label{typicalI}
E_d[e^{-\beta I}] = 1-\frac{1}{\delta_\Lambda}
\frac{\psi_\Lambda(\beta)}{\beta-\Phi_\Lambda(0)}.
\end{equation}
It follows that 
the mean idle period is thus equal to
\begin{equation}
E_d[I] = \frac{-\psi_\Lambda'(0+)}{\delta_\Lambda \Phi_\Lambda(0)},
\label{meanidle}
\end{equation}
where $\psi_\Lambda'(0+) = E_d(\Lambda_1)<0$.

\subsection{Rates}
\label{Rates}
A cycle of the process $Q$ is defined as the interval from the beginning of
a busy period until the beginning of the next busy period.
We therefore have
\begin{equation}
\label{meancycle}
\text{mean cycle length = } E_d[B+I] = \frac{1-\psi_\Lambda'(0+)}
{\delta_\Lambda \Phi_\Lambda(0)}.
\end{equation}
We can express the common rate, $\lambda$, of $N_g$ and $N_d$ as
the inverse of the mean cycle length:
\begin{equation}
\label{commonrate}
\lambda := E N_d(0,1) = E N_g(0,1) = \frac{1}{ E_d[B+I]}
= \frac{\delta_\Lambda \Phi_\Lambda(0)}{1-\psi_\Lambda'(0+)}.
\end{equation}


\subsection{The maximum over a busy period}
We now derive the $P_d$-distribution of $Q^*$. Letting $\alpha=\beta=0$ in
\eqref{Hformula} we obtain
\[
P_d(Q^* \le x) = 1 -\frac{1}{\delta_\Lambda W(x)},
\]
where $W(x) \equiv W^{(0)}(x)$ is defined through its
Laplace transform
\begin{equation}
\label{Wlap}
\int_0 ^\infty e^{-\theta x} W(x) dx 
= \frac{1}{\psi_\Lambda(\theta)}, \quad \text{for } \theta> \Phi_\Lambda(0).
\end{equation}
An immediate observation is that $\lim_{x \to 0} P_d (Q^* \le x) =0$,
since
$W(0) = \lim_{\theta \to 0} \theta/\psi_\Lambda(\theta)=1/\delta_\Lambda$.
So under $P_d$, the random variable $Q^*$  
has no atom at zero--which is, of course, expected. 

We now show that $Q^*$ has exponential tail under $P_d$ and derive
the precise asymptotics.
%
%
To do this, let 
\[
\beta^*:= \Phi_\Lambda(0).
\]
Then \eqref{Wlap} gives
that the Laplace transform of $x \mapsto e^{-\beta^* x} W(x)$ is 
$\theta \mapsto 1/\psi_\Lambda(\beta^*+\theta)$.
{From} the final value theorem for Laplace transforms,
\[
\lim_{x \to \infty} e^{-\beta^* x} W(x) = \lim_{\theta \to 0}
\frac{\theta}{\psi_\Lambda(\beta^*+\theta)} 
= \frac{1}{{\psi_\Lambda}'(\beta^*)},
\]
where we used the fact that $\psi_\Lambda(\beta^*)=0$.
It follows that
\[
 P_d (Q^* > x)\sim \frac{\psi_\Lambda'(\Phi_\Lambda(0))}{\delta_\Lambda} e^{-\Phi_\Lambda(0)x}
\]
as $x\rightarrow\infty$.

\section{Cycle formulae}
\label{cycleform}
We now show how the use of cycle formulae of Palm calculus
enable us to find (Proposition \ref{Iobs} below) the joint law of 
the endpoints of an idle period conditional on the event that
the idle period contains the origin of time.
Also (Proposition \ref{Bobs} below) we characterise the joint law of the
endpoints of a busy period, together with the maximum of $Q$ over this
busy period, conditional on the event that
the busy period contains the origin of time.

It is well-known that if $(\Omega, \FF, P)$ is endowed with
a $P$-preserving flow $(\theta_t, t \ge 0)$ (see end of Section \ref{Intro})
then for any random measure $M$ with finite intensity $\lambda_M$,
and any point process $N$ with finite intensity $\lambda_N$
such that $M(B, \theta_t\omega) = M(B+t, \omega)$,
$N(B, \theta_t\omega) = N(B+t, \omega)$, for $t \in \R$, $B$ Borel subset of $\R$,
and $\omega \in \Omega$, and any nonnegative measurable $Z :\Omega \to \R$, we have
\begin{equation}
\label{exch}
\lambda_M E_M[Z] 
=
\lambda_N E_N \int_{T_k}^{T_{k+1}} Z \comp \theta_{t} ~ M(dt),
\end{equation}
where $P_M$, $E_M$ (respectively, $P_N$, $E_N$) denotes Palm probability
and expectation with respect to $M$ (respectively, $N$), $T_0$ is
the first atom of $N$ which is $\le 0$, and $T_k, T_{k+1}$ 
are any two successive atoms of $N$.

The next result can be found for some special cases in \cite{KoSa} (diffusions), \cite{KKSS} (L\'evy processes), and  in \cite{Si} the general expression is derived. Here we offer a new proof in the general case based on (\ref{exch}).

\begin{proposition}[Joint law of endpoints of idle period]
\label{Iobs}
Assume that (A1)--(A5) hold. 
Then, conditional on $Q_0=0$, the left end-point, $g(0)$, and right
end-point, $d(0)$, of the idle period containing $t=0$ have
joint Laplace transform given by
\[
E[ e^{-\alpha d(0)+\beta g(0)} \mid Q_0 =0]
= \frac{\Phi_\Lambda(0)}{-\psi'_\Lambda(0+)} \cdot \frac{1}{\alpha-\beta}
\bigg(
\frac{\psi_\Lambda(\alpha)}{\alpha- \Phi_\Lambda(0)} 
-
\frac{\psi_\Lambda(\beta)}{\beta- \Phi_\Lambda(0)} 
\bigg) ,
\]
for non-negative $\alpha$ and $\beta$ ($\alpha\not=\beta$).
\end{proposition}
\begin{proof}
Let $M_I$ be the restriction of the Lebesgue measure on the idle
periods:
\[
M_I(A) = \int_A \1(Q_t=0) dt, \quad A \in \BB(\R).
\]
Then $E_{M_I}[Z] = E[Z | Q_0=0]$ for all nonnegative random variables $Z$.
Apply \eqref{exch} with $M=M_I$, $N=N_d$, and $Z=e^{-\alpha d(0)+\beta d(0)}$:
\[
\lambda_{M_I} E_{M_I}[e^{-\alpha d(0)+\beta g(0)}]
=
\lambda E_{d} \int_{d(-1)}^{d(0)} 
e^{-\alpha d(0) \comp \theta_t
+\beta g(0)\comp \theta_t}  ~ M_I(dt).
\]
Here $\lambda$ is the rate of $N_d$ and is given by \eqref{commonrate}.
The rate $\lambda_{M_I}$ is given by
\[
\lambda_{M_I} = \frac{E_d[I]}{E_d[B+I]} .
\]
Hence 
\[
\frac{\lambda}{\lambda_{M_I}} = \frac{1}{E_d[I]} 
= \frac{\delta_\Lambda \Phi_\Lambda(0)}{-\psi'_\Lambda(0+)},
\]
where we used \eqref{meanidle} and \eqref{meancycle}.
Now, $P_d(d(0)=0)=1$. To compute the integral above,
note that $M_I$ is zero on the interval $(d(-1), g(0))$,
and that, for $g(0) \le t \le 0$, we have $d(0)\comp \theta_t = -t$,
and $g(0)\comp\theta_t = g(0)-t$. So the integral above equals
\[
\int_{g(0)}^0 e^{(\alpha-\beta)t-\alpha g(1)}~ dt
= \frac{e^{\beta g(0)} - e^{\alpha g(0)}}{\alpha-\beta}.
\]
Combining the above we obtain
\[
E[ e^{-\alpha d(0)+\beta g(0)} \mid Q_0 =0]
=
\frac{\Phi_\Lambda(0)}{-\psi'_\Lambda(0+)} 
\cdot
\frac{E_d[e^{\beta g(0)}]-E_d[e^{\alpha g(0)}]}{\alpha-\beta}.
\]
Since $E_d[e^{\beta g(0)}] = E_d[e^{-\beta I}]$, the result is
obtained by using \eqref{typicalI}.
\QED
\end{proof}

\begin{proposition}[Joint law of endpoints of busy period and maximum over it]
\label{Bobs}
Assume that (A1)--(A5) hold.
Then, conditional on $Q_0>0$, the left end-point, $d(0)$, and right
end-point, $g(1)$, of the busy  period containing $t=0$, together
with the maximum of $Q_s$ for $s$ ranging over this busy period have
a joint law which is characterised by
\begin{eqnarray}
\label{max00}
\nonumber
\lefteqn{
E [e^{-\alpha g(1)+\beta d(0)} \1(Q^* \le x)\mid Q_0 >0]
}
\\[2mm]
&=& \frac{\Phi_\Lambda(0)}{\alpha-\beta}
\bigg(
\frac{1+\alpha\int_0^x W^{(\alpha)}(y)dy}{W^{(\alpha)}(x)}
-
\frac{1+\beta\int_0^x W^{(\beta)}(y)dy}{W^{(\beta)}(x)}
\bigg),
\end{eqnarray}
for  non-negative $\alpha$ and $\beta$ ($\alpha\not=\beta$).
\end{proposition}
\begin{proof}
Let $M_B$ be the restriction of the Lebesgue measure on the busy periods:
\[
M_B(A) = \int_A \1(Q_t>0) dt, \quad A \in \BB(\R).
\]

Then $E_{M_B}[Z] = E[Z|Q_0 > 0]$ for all random variables $Z \ge 0$.
Apply \eqref{exch}:
\begin{align}
\lefteqn{
\lambda_{M_B} E_{M_B} [ e^{-\alpha g(1)+\beta d(0)} \1(Q^* \le x)]
}
\nonumber \\
&= \lambda E_d \int_{d(0)}^{d(1)} 
e^{-\alpha g(1)\comp\theta_t +\beta d(0)\comp\theta_t } 
\1(Q^*\comp\theta_t  \le x) ~M_B(dt)
\nonumber \\
&= \lambda E_d \int_0^{g(1)} 
e^{-\alpha (g(1)-t) - \beta t} \1(Q^* \le x) ~ dt
\nonumber \\
&= \lambda E_d 
\bigg[
\1(Q^* \le x) ~ e^{-\alpha g(1)} ~
\frac{e^{(\alpha-\beta)g(1)}-1}{\alpha-\beta}
\bigg]
\nonumber \\
&= \frac{\lambda}{\alpha-\beta}
\bigg(
E_d[e^{-\beta g(1)} \1(Q^* \le x) ]
-
E_d[e^{-\alpha g(1)} \1(Q^* \le x) ]
\bigg)
\nonumber \\
&= \frac{\lambda}{\alpha-\beta}
(H(\beta,0,x) - H(\alpha, 0,x)),
\label{argument}
\end{align}
where $H(\alpha,\beta,x)$ is the right-hand side of
\eqref{triplelaw}. Using \eqref{commonrate}, 
\eqref{meancycle} and \eqref{busym}, we have
\[
\frac{\lambda}{\lambda_{M_B}} = \frac{1}{E_d[B]}
= 
\delta_\Lambda \Phi_\Lambda(0).
\]
Combining the above we obtain the announced formula.
\QED
\end{proof}

Proposition \ref{Bobs} yields the next corollary which recovers a result obtained in 
\cite{Si} using different methods (for special cases, see \cite{KoSa} and \cite{KKSS}). 
Clearly, Corollary \ref{cor0} could also be proved analogously as Proposition \ref{Iobs}.
\begin{corollary}
\label{cor0}
Assume that (A1)--(A5) hold.
Then, conditional on $Q_0>0$, the left end-point, $d(0)$, and right
end-point, $g(1)$, of the busy  period containing $t=0$ have joint Laplace transform given by
\[
E [e^{-\alpha g(1)+\beta d(0)} \mid Q_0 >0]
=
\frac{\Phi_\Lambda(0)}{\alpha-\beta} 
\cdot
\bigg(
\frac{\alpha}{\Phi_\Lambda(\alpha)}
-
\frac{\beta}{\Phi_\Lambda(\beta)}
\bigg),
\]
for non-negative $\alpha$ and $\beta$ ($\alpha\not=\beta$). 
\end{corollary}
\begin{proof}
The argument proceeds as in the proof Proposition \ref{Bobs}
by omitting the factor $\1(Q^* \le x)$, i.e.\ by formally
replacing $x$ with $+\infty$. The last line of \eqref{argument}
will give $\frac{\lambda}{\alpha-\beta}
(H(\beta,0,\infty) - H(\alpha, 0,\infty))$,
where $H(\alpha,\beta,\infty)$ is given by the right-hand side
of \eqref{no-max}.
\QED
\end{proof}

\begin{corollary}
Assume that (A1)--(A5) hold.
Then, conditional on $Q_0>0$, 
the maximum of $Q$ over the  busy  period containing $t=0$ 
has distribution
\begin{equation}
\label{max1}
P(Q^* \le x \mid Q_0 >0) = \Phi_{\Lambda}(0)\frac{W(x)\int_0^xW(y)dy - \int_0^x W(x-y)W(y)dy}{W(x)^2}
\end{equation}

for $x\geq 0$.
\end{corollary}

\begin{proof} Letting $\alpha,\beta\to 0$ in (\ref{max00}) yields
\begin{equation*}
P(Q^* \le x \mid Q_0 >0)  
= \Phi_{\Lambda}(0)\,\lim_{\alpha \to 0} \frac{\partial \hat H}{\partial\alpha}(\alpha,0,x),
\end{equation*}
where
\[
\hat H(\alpha, 0, x) 
= 1 - 
\frac{1+\alpha \int_0^x W^{(\alpha)}(y) dy}{ W^{(\alpha)}(x)}.
\]
Next recall that for each $x>0$, $W^{(\alpha)}(x)$ is an entire function in the variable $\alpha$ and in particular
\[
W^{(\alpha)}(x) = \sum_{k\geq 0} \alpha^k W^{*(k+1)}(x)
\]
where $W^{*(k+1)}(x)$ is the $(k+1)$-th convolution of $W$ (cf. Bertoin \cite{B97}). From this one easily deduces that 
$$
\frac{\partial}{\partial\alpha}W^{(\alpha)}(x)|_{\alpha=0}=\int_0^x W(y)W(x-y)dy.
$$
The result now follows from straightforward differentiation.
\QED
\end{proof}

\section{Example:  
Local time storage from reflected Brownian motion with negative drift}   
\label{example1}
Let $X = \{X_t, t \in \R\}$ be a reflected Brownian motion 
with drift $-c<0$  in stationary state
living on $I=[0, \infty),$ and let $P_0$ denote the probability 
measure associated with $X$ when initiated from 0 at time 0. 
Its local time (at 0) for $s<t$ is given by  
\begin{equation}
         L(s,t]:= \lim_{\varepsilon \searrow 0}
         \frac{1}{2\varepsilon}
         \int_s^t \1_{[0,\varepsilon)}
         (X_u) \,du.
         \label{eqloc}
    \end{equation}
Let $Q$ be the stationary process defined as in  (\ref{Q}):
$$
Q_t:=\sup_{s \leq t}\{L(s,t]-(t-s)\}.
$$
This particular example of fluid queues was introduced and analysed in 
\cite{{MNS}} and further 
studied in 
\cite{KoSa} and 
\cite{KKSS}.

Recall that $E_0L(0,1]=c,$ and, hence $Q$ 
is well-defined if and only if $0<c<\nolinebreak 1$. 
Here we make this example more complete by finding the 
$\alpha$-scale function associated with process 
$\Lambda_t:=t-L^{-1}_t, t\geq 0,$ where 
$$
L^{-1}_t:=\inf\{s: L(0,s] > t\}, \quad t \geq 0.
$$
is the inverse local time process. 
As seen from formulae (\ref{max1}) and (\ref{Hformula}),
the $\alpha$-scale function is the key ingredient needed for computing 
the distribution of the maximum of $Q$ over a busy period
and related random variables.

To begin with, we recall some basic formulae. When normalising as in (\ref{eqloc}),
see 
\cite[pg.\ 214]{itomckean74}, \cite[pg.\ 22]{BS}  it holds that
\begin{eqnarray} 
\label{levy}
 && 
\nonumber
 E_0 (\exp \{-\theta L^{-1}_t\}) = 
\exp \left\{ -t\int_0^\infty(1-{\rm e}^{-\theta u})\frac{1}{\sqrt{2\pi
    u^3}}\, {\rm e}^{-c^2u/2}\,du\right\}
\\
&&\hskip3.0cm
=\exp \left\{ -\frac{t}{G_\theta(0,0)} \right\},
\end{eqnarray} 
where 
$$
G_\theta(0,0)
:= \frac{1}{\sqrt{2 \theta+ c^2}-c}
$$
is the resolvent kernel (Green kernel) of $X$ at $(0,0)$;
see \cite[pg.\ 129]{BS}.
Consequently, we have   
\begin{equation*}
   E_0 \left(\exp\{\theta \Lambda_t\}\right) =
   \exp\left \{t\left(\theta-\frac{1}{G_\theta(0,0)}\right) \right\}
   = \exp \{t \psi_{\Lambda }(\theta) \},
   \end{equation*}
where
$$
\psi_{\Lambda }(\theta):= \theta-\sqrt{2 \theta+c^2} + c, \quad \theta \geq 0.
$$

Recall (cf. (\ref{alpha})) that the $\alpha$-scale function ($\alpha\geq 0$) associated with $\Lambda$ 
is defined for $x\geq 0$ via 
\begin{equation}
\label{s1}
\int_0^\infty {\rm e}^{-\theta x}W^{(\alpha)}(x)dx=\frac 1{ \psi_{\Lambda }(\theta)-\alpha}; 
\end{equation}
 for $x<0$ we set $W^{(\alpha)}(x)=0.$ The 0-scale function is called
simply the scale function and denoted $W.$ For the next proposition introduce
$$
{\rm Erfc}(x):=\frac 2{\sqrt{\pi}}\,\int_x^\infty {\rm e}^{-t^2}dt,
$$
and notice that ${\rm Erfc}(0)=1,\ {\rm Erfc}(+\infty)=0,$ and ${\rm Erfc}(-\infty)=2.$ 
\begin{proposition}
\label{scale1} 
The $\alpha$-scale function $W^{(\alpha)}$ of $\Lambda$ is for $x\geq 0$ given by
\begin{eqnarray}
\label{scale-alpha}
&&\hskip-1cm 
\nonumber
W^{(\alpha)}(x)=\frac {{\rm e}^{-c^2 x/2}}{\lambda_1-\lambda_2}
\Big(\lambda_1{\rm  e}^{\lambda_1^2\, x/2}
\,{\rm Erfc}(-\lambda_1\sqrt{x/2})
\\
&&
\hskip5cm
 -\lambda_2{\rm  e}^{\lambda_2^2\, x/2}
\,{\rm Erfc}(-\lambda_2\sqrt{x/2})\Big),
\end{eqnarray}
where
\begin{equation}
\label{roots}
\lambda_{1}:= 1+\sqrt{(1-c)^2+2\alpha},\quad
\lambda_{2}:= 1-\sqrt{(1-c)^2+2\alpha}
\end{equation} 
In particular, 
\begin{eqnarray}
\label{scalefor}
&&\hskip-1cm 
\nonumber
W(x)=\frac {{\rm e}^{-c^2 x/2}}{2(1-c)}\Big((2-c)\,{\rm e}^{(2-c)^2 x/2}
\,{\rm Erfc}(-(2-c)\sqrt{x/2})
\\
&&\hskip5cm
-c\,{\rm e}^{\,c^2x/2}\,{\rm
  Erfc}(-c\sqrt{x/2})\Big),
   \end{eqnarray}
   and $W(0)=1.$
\end{proposition}
\begin{proof} From (\ref{s1})  we have 
\begin{equation}
\label{lap20}
\int_0^\infty {\rm e}^{-\theta x}\,W^{(\alpha)}(x)dx= \frac 1{\theta-\sqrt{2 \theta+c^2} + c-\alpha}.
\end{equation}
To invert this Laplace transform, introduce $\lambda:= 2\theta+c^2.$
With this notation, 
\begin{eqnarray}
\label{lap21}
&&\hskip-1cm
\nonumber
\frac 1{\theta-\sqrt{2 \theta+c^2} + c-\alpha}=\frac
2{\lambda-2\sqrt{\lambda} +2(c-\alpha)-c^2}
\\
&&
\hskip2.4cm
=\frac 2{(\sqrt{\lambda}-\lambda_1)(\sqrt{\lambda}-\lambda_2)}
\\
&&
\nonumber
\hskip2.4cm
=\frac2{\lambda_1-\lambda_2}\left(\frac
1{\sqrt{\lambda}-\lambda_1}-\frac 1{\sqrt{\lambda}-\lambda_2}\right),
 \end{eqnarray}
where $\lambda_{1,2}$ are the roots of the equation
$z^2-2z+2(c-\alpha)-c^2=0,$ i.e., as in (\ref{roots}). 
Next, recall the following Laplace inversion formula
(cf.\ Erd\'elyi \cite[pg.\ 233]{erdelyi54})
\begin{equation}
\label{L0}
{\mathcal L}^{-1}\left(\frac 1{\sqrt{\lambda}+\beta}\right)=\frac
1{\sqrt{\pi x}}-\beta\,{\rm e}^{\beta^2 x}\,{\rm Erfc}(\beta\sqrt{x})
\end{equation}
valid for $\lambda-\beta^2>0.$ Since
\begin{eqnarray*}
&&\hskip-1cm
\int_0^\infty {\rm e}^{-\theta x}W^{(\alpha)}(x)dx=\int_0^\infty {\rm
  e}^{-\lambda y} {\rm
  e}^{c^2\, y}\,W^{(\alpha)}(2y)2dy
 \end{eqnarray*}
we obtain using (\ref{L0})
$$
2{\rm  e}^{c^2\, y}\,W^{(\alpha)}(2y)= \frac 2{\lambda_1-\lambda_2}\left(\lambda_1{\rm  e}^{\lambda_1^2\, y}
\,{\rm Erfc}(-\lambda_1\sqrt{y}) -\lambda_2{\rm  e}^{\lambda_2^2\, y}
\,{\rm Erfc}(-\lambda_2\sqrt{y})\right),
$$
which is formula (\ref{scale-alpha}). In particular, when $\alpha=0$ it holds $\lambda_1=2-c$ and
$\lambda_2=c$ yielding formula (\ref{scalefor}).
\QED
\end{proof}

Using the scale function $W$ and the fact 
$$
\Phi_\Lambda(0)=\sup\{\theta>0\,:\, \psi_\Lambda(\theta)=0\}=2(1-c)
$$
formula (\ref{max1}) yields  
the distribution of the maximum $Q^*$ over an observed busy period 
(i.e.\ over a busy period containing the origin of time).

\begin{proposition} 
Let $0 < c < 1$.
The distribution of the maximum $Q^*$ over an
observed  busy period of a reflected Brownian motion
(with drift $-c$) local time storage is given by
\begin{equation}
\label{L12} 
P(Q^*\leq x\mid Q_0 > 0)
=2(1-c)\frac{\int_0^x\,W(y)\left( W(x)-W(x-y)\right)\,dy}
{W^2(x)},
\end{equation}
where the scale function $W$ is given by (\ref{scalefor}). 
\end{proposition}

We plot the derivative of \eqref{L12} for $c=1/2$ in Figure \ref{fstar} below.
\begin{figure}[h]
\begin{center}
\includegraphics[width=6cm]{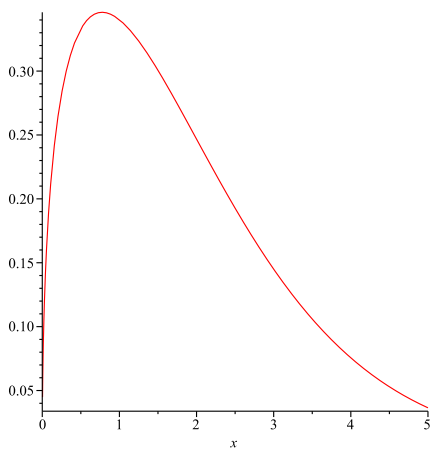}
\caption{\em The density of $Q^*$ conditional on $\{Q_0>0\}$ for
the example corresponding to Brownian motion with drift $-c=-1/2$.
}
\label{fstar}
\end{center}
\end{figure}

We recall some formulas from \cite{KoSa}. First 
\begin{eqnarray} \label{busybrown}
 &&E \left( e^{\theta d(0) -\beta g(1)}\,|\, Q_0>0 \right) 
\hspace{0.2cm} = \frac{8(1-c)}{\sqrt{2 \theta+(1-c)^2}+ \sqrt{2 \beta+(1-c)^2}}
 \nonumber  \\
 & &  \hspace{1.5cm} \times
 \frac{1}{(\sqrt{2 \theta+(1-c)^2}+1+c)(\sqrt{2 \beta+(1-c)^2}+1+c)} \nonumber \\
 & & \hspace{0.2cm} =:F(\theta, \beta; 1-c),
\end{eqnarray}

and
\begin{equation}
\label{idlebrown}
\E \left(e^{\theta g(0) -\beta d(0)}\,|\, Q_0=0 \right) =
F(\theta, \beta; c).
\end{equation}
Setting $\beta = \theta$ in the right-hand side
of (\ref{busybrown}) and (\ref{idlebrown}), respectively,  
 we get
\begin{equation} \label{lapvb}
  E \left( e^{-\theta (g(1) -d(0))} \,|\, Q_0>0 \right) =
  \frac{4 (1-c)}{\sqrt{2 \theta +(1-c)^2} (\sqrt{2 \theta +(1-c)^2}+1+c)^2},
\end{equation}
and 
\begin{equation} \label{lapvi}
  E \left( e^{-\theta (d(0) -g(0))} \,|\, Q_0=0 \right) =
  \frac{4 c}{\sqrt{2 \theta +c^2} (\sqrt{2 \theta +c^2}+2-c)^2},
\end{equation}
Taking the inverse Laplace transform of $(\ref{lapvb})$ (cf.
Erd\'elyi \cite[pg.\ 234]{erdelyi54})
we obtain the density of the length of the busy period $g(1) - d(0)$, 
given $Q_0>0$, as
$$
f_{g-b} (v) = 2(1-c){\rm e}^{-(1-c)^2v/2}\left( \sqrt{2v/\pi}-  (1+c) v
e^{(1+c)^2v/2}{\rm Erfc}((1+c)\sqrt{v/2})\right)
$$
Note that the density of the length of the idle period $d(0) - g(0)$,
given that $Q_0=0$, is obtained from $f_{g-b}(v)$ by substituting $c$ for $1-c$.
In Figure \ref{fgb} we have ploted $f_{g-b} (v)$ for three different
values of $c$.
\begin{figure}[h]
\begin{center}
\includegraphics[width=7cm]{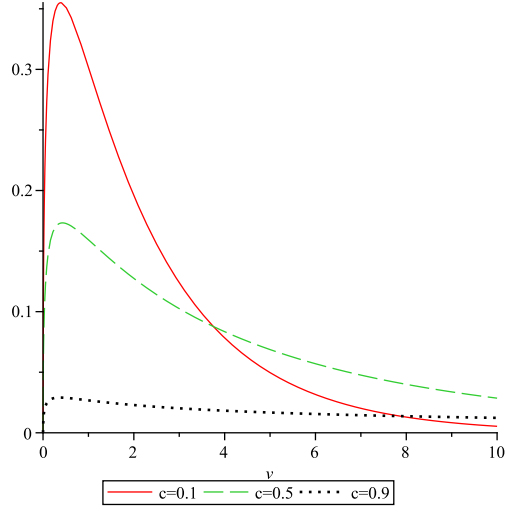}
\caption{\em The density of length of the busy period, given that $Q_0>0$
for three different values of $c$.}
\label{fgb}
\end{center}
\end{figure}
We notice also that the mean busy period length has a simple expression:
\[
E[g(1)-d(0) \mid Q_0 > 0] = \frac{2-c}{(1-c)^2}.
\]
The joint density of $d(0)$ and $g(1)$ is given by 
  \begin{eqnarray*}
  && \P(-d(0) \in dx, g(1) \in dy \,|\, Q_0>0)=
   2(1-c){\rm e}^{-(1-c)^2(x+y)/2}
\\
&&
\hskip1cm
\times\left( \sqrt{2/(\pi(x+y))}-  (1+c)
e^{(1+c)^2(x+y)/2}{\rm Erfc}((1+c)\sqrt{(x+y)/2})\right)
\end{eqnarray*}
and, again, the density for $(g(0),d(0))$ is obtained by substituting $c$
for $1-c.$ 

Next we find the density of $g(1)$ (recall that $-d(0)$ is identical 
in law with $g(1)$)  by inverting the Laplace transform
(obtained from (\ref{busybrown}) by choosing $\theta=0$):
\begin{eqnarray} 
\label{busyd}
\nonumber
 &&\hskip-1cm
E \left(e^{-\beta g(1)}\,|\, Q_0>0 \right) 
\\
&&
= \frac{4(1-c)}{\left(\sqrt{2 \beta+(1-c)^2}+ 1-c\right)
\left(\sqrt{2 \beta+(1-c)^2}+1+c\right)}.
 \end{eqnarray}
Letting  $\lambda:=2 \beta+(1-c)^2$ we rewrite (\ref{busyd})
as 
\begin{eqnarray} 
\label{busyd2}
\nonumber
 &&\hskip-1cm
E \left(e^{-\beta g(1)}\,|\, Q_0>0 \right) 
= \frac{2(1-c)}{c}\left(\frac{1}{\sqrt{\lambda}+ 1-c}-\frac{1}{\sqrt{\lambda}+ 1+c}\right).
 \end{eqnarray}
{From} (\ref{L0}) 

\begin{align} 
\lefteqn{
{\mathcal L}^{-1}\left(\frac{1}{\sqrt{\lambda}+ 1-c}
-\frac{1}{\sqrt{\lambda}+
  1+c}\right)}
\nonumber \\
& =(1+c)\,{\rm e}^{(1+c)^2 x}\,{\rm Erfc}\left((1+c)\sqrt{x}\right)
-(1-c)\,{\rm e}^{(1-c)^2 x}\,{\rm Erfc}\left((1-c)\sqrt{x}\right).
\label{L01}
\end{align} 
Consequently,
\begin{multline} 
2 {\rm e}^{-(1-c)^2x}f_{g(1)}(2x)
= \frac{2(1-c)}{c}
\Big((1+c)\,{\rm e}^{(1+c)^2 x}\,{\rm Erfc}((1+c)\sqrt{x})
\\
-(1-c)\,{\rm e}^{(1-c)^2 x}\,{\rm Erfc}((1-c)\sqrt{x})\Big),
\label{den2}
\end{multline}
where $f_{g(1)}$ denotes the density of $g(1)$ conditioned on $\{Q_0>0\}. $ 
From (\ref{den2}) we obtain
\begin{multline} 
 f_{g(1)}(x)=\frac{(1-c)\,{\rm e}^{-(1-c)^2x/2}}{c} ~
\Big((1+c)\,{\rm e}^{(1+c)^2 x/2}\,{\rm Erfc}\left((1+c)\sqrt{x/2}\right)
\\
-(1-c)\,{\rm e}^{(1-c)^2 x/2}\,{\rm Erfc}\left((1-c)\sqrt{x/2}\right)\Big).
\label{den3}
\end{multline}
Moreover, the density  $f_{d(0)}$ of $d(0)$ conditional on 
$\{Q_0=0\}$ is obtained from (\ref{den3}) by
substituting $c$ instead of $1-c:$
\begin{multline} 
f_{d(0)}(x)=\frac{c\,{\rm e}^{-c^2x/2}}{1-c}~
\Big((2-c)\,{\rm e}^{(2-c)^2 x/2}\,{\rm Erfc}\left((2-c)\sqrt{x/2}\right)
\\
-c\,{\rm e}^{c^2 x/2}\,{\rm Erfc}\left(c\sqrt{x/2}\right)\Big).
\label{den4}
\end{multline}
It is striking how similar formulae (\ref{scalefor}) and (\ref{den4}) are. 
\begin{remark} 
The scale function formulae (\ref{scale-alpha}) and (\ref{scalefor}) are clearly valid for all $c\geq 0.$  In case $c=0$ the process $\{L(0,t]\,;\, t\geq 0\}$ is a version of the Brownian local time, and the $\alpha$-scale function $ W_0^{(\alpha)} $ of the corresponding process $\Lambda$ is given by 
\begin{multline*}
W_0^{(\alpha)}(x)=\frac{{\rm e}^{(1+\alpha)x}}{2\sqrt{1+2\alpha}}\Big(
(1+\sqrt{1+2\alpha})\,{\rm e}^{x\sqrt{1+2\alpha}}\,{\rm Erfc}\left(-(1+\sqrt{1+2\alpha})\,\sqrt{x/2}\right)
\\
-(1-\sqrt{1+2\alpha})\,{\rm e}^{-x\sqrt{1+2\alpha}}\,{\rm Erfc}\left(-(1-\sqrt{1+2\alpha})\,\sqrt{x/2}\right)
\Big).
\end{multline*} 
In particular,
\begin{equation}
\label{rbm1}
W_0(x)={\rm e}^{2x}\,{\rm Erfc}(-\sqrt{2x}).
\end{equation}
In case $c=1$ it holds
$\lambda_{1,2}=1\pm\sqrt{2\alpha}$ and for $\alpha\not= 0$ formula (\ref{scalefor}) can be used directly. For the 0-scale function we need to take the limit as $c\to 1$ in (\ref{scalefor}): 
\begin{equation}
\label{p111}
W_1(x)=(1+x)\,{\rm Erfc}(-\sqrt{x/2})+\sqrt{\frac{2x}{\pi}}\,{\rm e}^{-x/2}.
\end{equation}
\end{remark}
\begin{remark} 
Here we display some formulae for 
Laplace transforms apparent from above and point out a misprint in Erd\'elyi et al. \cite{erdelyi54}.

First,
from (\ref{lap20}),  (\ref{lap21}), and (\ref{p111}) 
we have the following Laplace inversion
formula valid for $\lambda>1$:
\begin{equation}
\label{p1}
{\mathcal L}^{-1}\left(\frac 1{(\sqrt{\lambda}-1)^2}\right)=
(1+2x)\,{\rm e}^x\,{\rm Erfc}(-\sqrt{x})+\frac{2\sqrt x}{\sqrt\pi}, 
\end{equation}
and this can be ``extended''  (for $a>0$) to
\begin{equation}
\label{p2}
{\mathcal L}^{-1}\left(\frac 1{(\sqrt{\lambda}-a)^2}\right)=
(1+2a^2x)\,{\rm e}^{a^2x}\,{\rm Erfc}(-a\sqrt{x})+\frac{2a\sqrt x}{\sqrt\pi}. 
\end{equation}
Furthermore, it can be checked that (\ref{p2}) is valid for all $a<0$
by evaluating the Laplace transform of the right-hand side. 
This can be done
term by term  by using, e.g., Erd\'elyi et al. \cite[pp.\ 137, 177]{erdelyi54} 
(well-known formulae): 
$$
{\mathcal
  L}\left(\sqrt{x}\right)=\frac{\sqrt{\pi}}2\,\lambda^{-3/2},
$$
$$
{\mathcal L}\left({\rm e}^{a^2 x}\,{\rm
  Erfc}(a\sqrt{x})\right)=\lambda^{-1/2}\,(\lambda^{1/2}+a)^{-1},
$$
and 
\begin{eqnarray*}
&&
{\mathcal L}\left(x{\rm e}^{a^2 x}\,{\rm
  Erfc}(a\sqrt{x})\right)=-\frac{\partial}{\partial \lambda}
{\mathcal L}\left({\rm e}^{a^2 x}\,{\rm
  Erfc}(a\sqrt{x})\right)
\\
&&
\hskip3.3cm
=-\frac{\partial}{\partial \lambda}\lambda^{-1/2}\,(\lambda^{1/2}+a)^{-1}
\\
&&
\hskip3.3cm
=\frac 1{2a}\left(\lambda^{-3/2} -\lambda^{-1/2}\,(\lambda^{1/2}+a)^{-2}\right).
\end{eqnarray*}
We remark that formula (10)
in \cite{erdelyi54} p. 234:
\begin{equation}
\label{e1}
{\mathcal L}^{-1}\left(\frac 1{(\sqrt{\lambda}+\sqrt{b})^2}\right)=
1-2\sqrt{bx/\pi}+ (1-2bx)\,{\rm e}^{bx}\left({\rm Erf}(\sqrt{bx})-1\right).
\end{equation}
is {\bf not} correct since it does not coincide with formula 
(\ref{p2}) (for $a<0$). Indeed, because 
$$
{\rm Erf}(x):=\frac 2{\sqrt{\pi}}\,\int_0^x {\rm e}^{-t^2}dt,
$$
the right-hand side of (\ref{e1}) is zero at zero but the right-hand side of
(\ref{p2}) is 1 at zero. 
\end{remark}
  
\section{Further examples}
In the previous example we derived a local time process from a given Markov process. However, it is also possible to consider examples where just the local time process $L$, or equivalently the subordinator $L^{-1}$, is specified. 
Indeed the subordinator that will play the role of 
$L^{-1}$ in this example has no drift and has L\'evy measure given by
\[
\Pi(x,\infty) = \frac{ \gamma^\nu}{\Gamma(\nu)}x^{\nu-1}e^{-\gamma x} + \varphi \frac{ \gamma^\nu}{\Gamma(\nu)}\int_x^\infty 
y^{\nu-1}e^{-\gamma y} dy,
\]
where the constants $\varphi,  \gamma>0$ and $\nu\in(0,1)$. Note in particular then that $L^{-1}$ is the sum of two independent subordinators, one of which is a compound Poisson process with gamma distributed jumps, the other has infinite activity  and is of the so called tempered-stable type.
Clearly $\Pi$ also describes the L\'evy measure of $-\Lambda$ too.

According to 
\cite{HK}, the process $\Lambda$ belongs to the Gaussian Tempered Stable Convolution class and moreover, 
\[
\psi_\Lambda(\theta) = (\theta  - \varphi)\left(1 - \left(\frac{\gamma}{\gamma+\theta}\right)^\nu\right)
\]
for $\theta\geq 0$. In particular $\delta_\Lambda =1$ and $\Phi_\Lambda(0) = \varphi$. It is a straightforward exericse to show that  
\[
\mathbb{E}(\Lambda_1)  = \psi_\Lambda'(0+) = -\varphi  \frac{\nu}{\gamma}
\]
and this implies that
\[
 \mu = \frac{1}{1+ \varphi\nu/\gamma}<1,
\]
as required.

From \cite{HK} we also know that
\[
W(x) = e^{\varphi x}+ \gamma^\nu e^{\varphi x}\int ^x e^{-(\gamma + \varphi)y} y^{\nu-1}{\rm E}_{\nu,\nu}(\gamma^\nu y^\nu)dy
\]
where 
\[
 {\rm E}_{\alpha, \beta}(x) := \sum_{n\geq 0}\frac{z^n}{\Gamma(\alpha n + \beta)}
\]
 is the two parameter Mittag-Leffler function.

We may now deduce from the theory presented earlier that, for example,
\[
P_d(Q^*\leq x) = \frac{1-e^{-\varphi x}+ \gamma^\nu \int ^x e^{-(\gamma + \varphi)y} y^{\nu-1}{\rm E}_{\nu,\nu}(\gamma^\nu y^\nu)dy
}{1+ \gamma^\nu \int ^x e^{-(\gamma + \varphi)y} y^{\nu-1}{\rm E}_{\nu,\nu}(\gamma^\nu y^\nu)dy
}
\]
and 
\[
 P_d(Q^*>x)\sim \left(1 - \left(\frac{\gamma}{\gamma+\varphi}\right)^\nu\right)e^{-\varphi x}.
\]
\vskip.5cm
\noindent
{\bf Acknowledgement.} We thank Ilkka Norros for posing the problem for finding the distribution of the maximum and Andrey Borodin for discussions on Laplace transforms.

\end{document}